\documentclass{amsart}
\usepackage{latexsym}
\usepackage{amsmath,amsfonts,epsfig, graphics, graphicx, amsthm, amssymb}
\input{xy}
\xyoption{all}

\newtheorem{theorem}{Theorem}[section]
\newtheorem{lemma}[theorem]{Lemma}

\newtheorem{varthmC}{Conjecture}
\newtheorem{varthm1}{Main Theorem}

\begin{document}
\title{Carmichael Numbers with Prime Numbers of Prime Factors}
\author{Thomas Wright}
\maketitle

\begin{abstract} Under the assumption of Heath-Brown's conjecture on the first prime in an arithmetic progression, we prove that there are infinitely many Carmichael numbers $n$ such that the number of prime factors of $n$ is prime. \end{abstract}

\section{Introduction}

Recall that a Carmichael number is a composite number $n$ for which

$$a^n\equiv a\pmod n$$
for every $a\in \mathbb Z$.

Although the set of Carmichael numbers was proven to be infinite in 1994 in a paper by Alford, Granville, and Pomerance \cite{AGP}, there are still many open conjectures about Carmichael numbers.   Chief among those conjectures is the following:

\begin{varthmC}  For any fixed $R\in \mathbb N$ with $R\geq 3$, there exist infinitely many Carmichael numbers with exactly $R$ prime factors.
\end{varthmC}

In fact, specific conjectures \cite{GP} have been made about the number of Carmichael numbers up to $x$ with specific numbers of prime factors\footnote{Interestingly, this behavior does not seem to hold if $R$ is large relative to $x$; Granville and Pomerance conjectured that if $R=\log^{\nu+o(1)} x$ where $0<\nu<1$ then there are $x^{\nu+o(1)}$ Carmichael numbers up to $x$ with exactly $R$ prime factors once $x$ is sufficiently large (see Conjecture 2a of \cite{GP}).}:

\begin{varthmC} (Granville-Pomerance, 1999) For any fixed $R\in \mathbb N$ with $R\geq 3$, let $C_R(x)$ denote the number of Carmichael numbers up to $x$ with exactly $R$ prime factors.  Then
$$C_R(x)=x^{\frac 1R+o(1)}.$$
\end{varthmC}

It is also known that there exist Carmichael numbers with $R$ prime factors where $R$ is any number between 3 and 10,333,229,505 \cite{AGHS}.

Unfortunately, little has been proven unconditionally about the number of Carmichael numbers with given numbers of prime factors.  As such, we have generally turned to conditional results involving the assumption of rather strong conjectures.  In this paper, we push those conditional results further, using a conjecture of Heath-Brown that has recently been applied to Carmichael numbers in a couple of different papers.







\subsection{Conditional Results}


In 1939, J. Chernick \cite{Ch} noted that under the assumption of Dickson's prime $k$-tuples conjecture, one could prove that there are infinitely many Carmichael numbers with exactly $R$ prime factors for any $R\geq 3$.  This insight came from the fact that one could construct tuples, such as the triple $6k+1$, $12k+1$, and $18k+1$, where any occurrence of all elements of the tuple being prime simultaneously would lead to a Carmichael number with exactly that many prime factors.

In papers \cite{WrF} and \cite{WrEv}, the current author was able to weaken the requirement of Dickson's full conjecture in proving the infinitude of Carmichael numbers with bounded numbers of prime factors; however, this weakened version of the theorem no longer allows us to prove Carmichael numbers with a specific number of prime factors.

In this paper, we tackle a slightly different question; rather than asking whether we can bound the number of prime factors, we wish to determine whether one can ensure that the number of prime factors is in a specific set.  To do so, however, we will need to invoke a conjecture of Heath-Brown's that has often been used in papers about Carmichael numbers:






\begin{varthmC} (Heath-Brown, 1978): For any $(c,d)=1$, the smallest prime that is congruent to $c$ (mod $d$) is $\ll d\log^2 d$.\end{varthmC}

This conjecture has been applied to Carmichael numbers several times recently, including \cite{EPT}, \cite{BP}, and \cite{WrV}; in some cases, it has been used to prove results that later went on to be proven unconditionally \cite{Ma}, \cite{WrE}, \cite{WrA}.  Here, we show that the assumption of this conjecture allows us to prove a statement about Carmichael numbers with prime numbers of prime factors.  As is standard, we will let $\Omega(n)$ denote the number of prime factors of $n$.

\begin{varthm1} If Heath-Brown's conjecture is true then there are infinitely many Carmichael numbers $n$ for which $\Omega(n)$ is prime.\end{varthm1}

We note that the conjecture we use here is actually quite a bit stronger than necessary; it would actually be sufficient to replace the $d\log^2 d$ with something like $de^{\sqrt{\log d}}$.  However, since we are not computing anything particularly sharply in this paper, and since even this weakened conjecture is nowhere close to being proven, we have decided to work instead with the more widely recognized conjecture since it makes for cleaner exposition.

\section{Introduction: Outline}
In order to prove Main Theorem 1, we begin by following the methods commonly used to construct an infinitude of Carmichael numbers as were first laid out in \cite{AGP}.  First, we find an $L$ such that $L$ has a large number of prime factors $q$ where each $q-1$ is relatively smooth.  Having done this, we then prove that there must exist a $k$ such that there are many primes $p=dk+1$ with $d|L$.  From here, our goal will be to find a subset of these primes $p_1,\cdots, p_r$ such that the product $n=p_1\cdots p_r$ is 1 modulo $L$ (and obviously 1 modulo $k$); since we will then have $p-1|Lk|n-1$ for each prime $p|n$, this product is then Carmichael number.

Aiding us in this quest is a combinatorial result that appears in \cite{AGP}, although this result is simply an application of a theorem that appears in \cite{Me} and \cite{EK}.  For an abelian group $G$, let $n(G)$ be the smallest number such that any sequence of length $n(G)$ must contain a subsequence whose product is the identity in $G$.  Theorem 2.1 in \cite{AGP} then states the following:

\begin{theorem}[Alford, Granville, Pomerance 1994]\label{AGP2.1} Let $G$ be an abelian group, and let $r>t>n=n(G)$ be integers.  Then any sequence of $r$ elements in $G$ must contain $\left(\begin{array}{c} r \\ t \end{array}\right)/\left(\begin{array}{c} r \\ n \end{array}\right)$ distinct subsequences of length at most $t$ and at least $t-n$ whose product is the identity.
\end{theorem}

By pigeonhole principle, there must then be an $h$ with $t\geq h\geq t-n$ such that there are many sequences of length $h$ whose product is the identity.  If we took $g$ distinct such sequences of length $h$, the product of these $g$ sequences would then yield a sequence of length $gh$ whose product is also the identity.  We can then use Heath-Brown's conjecture to append one more prime onto these sequences, giving us sequences of length $gh+1$ whose product is the identity as well.  We know that there must exist a relatively small $g$ for which $gh+1$ is prime ($g\ll \log^2 h$ under Heath-Brown's conjecture, $g\ll h^{4.2}$ unconditionally by [Xy]); thus, we can use this construction to give a Carmichael number with a prime number of factors.




\section{Constructing an $L$}

The first section follows roughly the same blueprint as \cite{AGP} and subsequent Carmichael papers.  Let $1<\theta<2$, and let $P(q-1)$ denote the size of the largest prime divisor of $q-1$.  For some choice of $y$, let us define the set $\mathcal Q$ by
\[\mathcal Q=\{q\mbox{ }prime:\frac{y^\theta}{\log y}\leq q\leq y^{\theta},\mbox{ }P(q-1)\leq y\}.\]

We cite the following lemma, the proof of which appears in many places including Lemma 5.1 of \cite{WrE} or Lemma 2.1 of \cite{WrA}:
\begin{lemma}\label{lemtwo}  Let $\theta\in (1,2)$.  For $\mathcal Q$ as above, there exist constants $\gamma=\gamma_\theta$ and $Y=Y_\theta$ such that

\[|\mathcal Q|\geq \gamma \frac{y^{\theta}}{\log (y^\theta)}\]
if $y>Y$.
\end{lemma}



From this, we let
\[L'=\prod_{q\in \mathcal Q}q.\]
Let $B=\frac{5}{12}$, and let $\pi(x,q,a)$ denote the number of primes $p\leq x$ with $p\equiv a\pmod q$.  For a given $x$, we know that for any $d|L'$ with $d\leq x^B$,

\begin{gather}\label{BT}\pi(x,d,1)\geq \frac{\pi(x)}{2\phi(d)}\end{gather}
as long as $d$ does not have a divisor in some small exceptional set $\mathcal D_B(x)$ (see e.g. \cite{AGP}, bottom of page 2 and top of page 3).  Here, it is known that $|\mathcal D_B(x)|$ is bounded by some constant $D_B$ determined only by $B$.  So for each $d\in \mathcal D_B(x)$, we choose a prime divisor $q|d$, and we let $\mathcal P_B(x)$ be the collection of such divisors.  Then we can define

\[L=\prod_{q\in \mathcal Q,q\not\in \mathcal P_B(x)}q.\]

With this alteration, we can now ensure that any divisor $d|L$ with $d\leq x^B$ must satisfy (\ref{BT}).

\section{Bounds for primes in arithmetic progressions}\label{section4}
Next, define

$$\mathcal P_k=\{p=dk+1:p\mbox{ prime}, d|L,\left(k,L\right)=1\}.$$

Our goal in this section will be to prove that, for some value of $k$, $\mathcal P_k$ is relatively large.  For this particular choice of $k$, $\mathcal P_k$ will then comprise the prime factors of our Carmichael numbers.  This section follows the standard framework laid in Section 4 of \cite{AGP}.

First, we know by  Montgomery and
Vaughan's explicit version of the Brun-Titchmarsh theorem \cite{MV} that if $z>h^2$ then
\begin{gather}\label{BT4}
\pi(z,h,1)\leq \frac{4z}{\phi(h)\log z}.
\end{gather}
By our choice of $L$, we can assume that the sum 
\begin{gather}\label{qsum}\sum_{\stackrel{q|L}{q\mbox{ }prime}}\frac{1}{q-1}\leq \frac{1}{32}.
\end{gather}
Moreover, if $d<x^B$ then
\begin{gather}\label{logs}
(1-B)\log x\leq \log\left(dx^{1-B}\right)\leq 2(1-B)\log x.
\end{gather}
Putting (\ref{BT})-(\ref{logs}) together, we find that the number of primes $p\leq dx^{1-B}$ with $p\equiv 1\pmod d$ and $(\frac{p-1}{d},L)=1$ is
\begin{align*}
\geq &\pi  (dx^{1-B},d,1)-\sum_{\substack{q|\frac Ld \\ q\mbox{ }prime}}\pi(dx^{1-B},dq,1)\\
\geq &\frac{dx^{1-B}}{4(1-B)\phi(d)\log x}-\sum_{\substack{q|\frac Ld \\ q\mbox{ }prime}}\frac{4dx^{1-B}}{(1-B)\phi(dq)\log
x}\\
\geq &\frac{dx^{1-B}}{4(1-B)\phi(d)\log x}-\frac{dx^{1-B}}{8(1-B)\phi(d)\log
x}\\
\geq &\frac{x^{1-B}}{8(1-B)\log x}.
\end{align*}
Hence,
\begin{align*}
\sum_{k\leq x^{1-B}}|\mathcal P_k|\geq &\sum_{d|L,d\leq x^{B}}\#\bigg{\{}p\leq dx^{1-B}:p\equiv 1\pmod d,\left(\frac{p-1}{d},L\right)=1\bigg{\}}\\
\geq & \sum_{d|L,d\leq x^{B}}\frac{x^{1-B}}{8(1-B)\log x}
\end{align*}
As we are under no imperative to make $x$ or $k$ small, we can simply take $x^{B}=L$ and find that
$$\sum_{k\leq x^{1-B}}|\mathcal P_k|\geq \frac{2^{\gamma \frac{y^\theta}{\log y}}x^{1-B}}{8\log x}.$$
Since there are $x^{1-B}$ choices for $k$ in the sum above, we know by pigeonhole principle that there must exist a $k_0$ for which
$$\left|\mathcal P_{k_0}\right|\geq \frac{2^{\gamma \frac{y^\theta}{\log y}}}{8(1-B)\log x}.$$

\section{Small Orders mod $L$}

In this section, we will show that the order of an element mod $L$ is very small relative to the size of the set $\mathcal P_{k_0}$ above.  Let $\lambda(L)$ denote the Carmichael lambda function, which signifies the largest order of an element mod $L$.  In this case, if we index the $q\in \mathcal Q$ as $q_1,\cdots, q_s$, we have
$$\lambda(L)=lcm\left(q_1-1,\cdots q_s-1\right).$$
Since all of the $q$ that divide $L$ are chosen such that $q-1$ is $y$-smooth, we know that $\lambda(L)$ must be free of prime factors of size greater than $y$.  Moreover, if $r^j|\lambda(L)$ for some prime $r$ then $r^j$ must also divide $q-1$ for some $q\leq \mathcal Q$, which means $r^j\leq q\leq y^\theta$.  Thus, for a given prime $r$, let $a_r$ be the largest power of $r$ such that $r^{a_r}\leq y^\theta$.  Then
\begin{gather}\label{lambdabd}
\lambda(L)\leq \prod_{\substack{r\leq y \\ r\mbox{ }prime}}r^{a_r}\leq \prod_{\substack{r\leq y \\ r\mbox{ }prime}}y^{\theta}\leq y^{\theta \pi(y)}\leq e^{2\theta y}.
\end{gather}
This is small relative to $x$, $L$, and $\left|\mathcal P_{k_0}\right|$.  Indeed,
\begin{gather}\label{L}
L\leq \prod_{q\in \mathcal Q}q\ll \prod_{q\in \mathcal Q}y^\theta\leq y^{2\theta\gamma \frac{y^\theta}{\log y}}=e^{2\theta\gamma y^\theta}
\end{gather}
and hence since $x^B=L$,
$$\left|\mathcal P_{k_0}\right|\geq \frac{2^{\gamma \frac{y^\theta}{\log y}}}{4\log x}\gg \frac{2^{\gamma \frac{y^\theta}{\log y}}}{y^\theta}\gg e^{\frac{\gamma y^\theta}{2\log y}}.$$

\section{The Extra Prime}\label{extra}
In Section \ref{Constructing},  we will construct a Carmichael number $n$; in doing this, we must be able to ensure that the number of prime factors of $n$ is 1 modulo a chosen variable $h$.  To do this, we introduce a conveniently constructed prime that we can often affix to the end of a Carmichael number coming from our construction to yield another Carmichael number.

In general, the \cite{AGP} construction of Carmichael numbers will require that $p-1|Lk_0$ and that the product of the primes $p|n$ will give $Lk_0|n-1$.  As such, it would be most convenient if we could find a prime $P=Lk_0+1$; after all, this $P$ would trivially satisfy $P-1|Lk_0$, and multiplying $P$ to any $n$ for which $Lk_0|n-1$ would yield $Lk_0|Pn-1$.  Unfortunately, it is difficult to guarantee that such a prime exists, so instead, we use the conjecture of Heath-Brown to find a prime that is almost as convenient.  In particular, we have the following:

\begin{lemma} Let $L$ be as defined above, and $k_0$ be as in Section \ref{section4}.  Under the assumption of Conjecture 3, there exists a $k_1\ll \log^2 L$ such that $P=Lk_0k_1+1$ is prime.\end{lemma}

The proof merely requires us to take $d=Lk_0$ and note that $k_0<L^3$.

\section{Sizes of Sets}\label{Constructing}

Now, we will need to find subsets of $\mathcal P_{k_0}$ whose products are 1 mod $Lk_0k_1$.  For a given $h$, let $C_h$ be the set of distinct subsets of exactly $h$ elements in $\mathcal P_{k_0}$ such that the product of all the elements in any of these subsets is indeed 1 mod $Lk_0k_1$.  We can then prove the following:


\begin{theorem} There must exist an $h$ with $n\leq h\leq 2n$ such that
$$|C_h|\gg \left(\frac{n}{4}\right)^n.$$
\end{theorem}

\begin{proof}
This is an application of Theorem 2.1.  In our case, since the group $G$ is the integers mod $Lk_1$ and $k_1\ll \log^2(Lk_0)$, we note that
\begin{align*}
\lambda(Lk_1)\leq &\lambda(L)\cdot k_1\\
\ll &e^{2\theta y}\log^2(Lk_0)\\
\ll &e^{3\theta y}
\end{align*}
where we can bound $Lk_0$ by (\ref{L}) and $\lambda(L)$ by (\ref{lambdabd}).  So
\[n(G)\ll e^{3\theta y}.\]
Now, following the notation in Theorem 2.1, the computations here will be cleanest if $t-n$ is close to $n$ and $r$ is roughly $t^2$.  So let
$$r=4n^2\ll e^{6\theta y}$$
and
$$t=2n$$
Recalling that 
$$\left(\frac{v}{w} \right)^{w}\leq \left(\begin{array}{c} v \\ w \end{array}\right)\leq \left(\frac{ve}{w} \right)^{w},$$
we can invoke Theorem \ref{AGP2.1} to find
\begin{align*}\sum_{h=n}^{2n}|C_h|\gg &\left(\begin{array}{c} 4n^2 \\ 2n \end{array}\right)/\left(\begin{array}{c} 4n^2 \\ n \end{array}\right) \\
&\gg \left(\frac{4n^2}{2n} \right)^{2n}/\left(\frac{4en^2}{n}\right)^n \\
&\gg \left(2n\right)^{2n}/\left(12n\right)^n\\
&\gg \left(2n\right)^{n}/\left(6\right)^n\\
&=\left(\frac{n}{3}\right)^n
\end{align*}
So there must exist an $h$ between $n$ and $2n$ for which
$$|C_h|\gg \frac{\left(\frac{n}{3}\right)^n}{n+1}=\frac{\left(\frac{n}{4}\right)^n\left(\frac{4}{3}\right)^n}{n+1}.$$
Since $$\left(\frac 43\right)^n\gg n+1,$$ the theorem then follows.
\end{proof}

From here, we can prove the theorem:

\begin{theorem}
For our choice of $y$, $\theta$, and $L$, there must exist a Carmichael number with a prime number of prime factors.
\end{theorem}
Since we have infinitely many choices for $y$, the proof of this theorem will complete the proof of Main Theorem 1.
\begin{proof}
Let us take a set of primes $S\in C_h$.  Then
$$n=\prod_{p\in S}p\equiv 1\pmod{Lk_0k_1},$$
where the congruence mod $Lk_1$ is by definition of $C_h$, and all of the $p$ in $\mathcal P_{k_0}$ are such that $p\equiv 1\pmod{k_0}$.

Similarly, if we take sets $S_1,\cdots,S_j\in C_h$ then
$$n=\prod_{p\in S_1\bigcup\cdots \bigcup S_j}p\equiv 1\pmod{Lk_0k_1},$$
since the product of a bunch of 1's is still 1.

Note, then, that for any $g$ that is significantly smaller than $\left(\frac n4\right)^n$ (say $g\ll \left(\frac n4\right)^{\frac n2}$), we can take a collection of $g$ sets $S_1,\cdots,S_g$ such that
$$n=\prod_{p\in S_1\bigcup\cdots \bigcup S_g}p\equiv 1\pmod{Lk_0k_1},$$
where $n$ has exactly $gh$ distinct prime factors.

Recall also that in Section \ref{extra}, we defined an ``extra" prime $P$ such that
$$P=Lk_0k_1+1.$$

We now wish to find a prime number of the form $gh+1$ where $g$ is small.  While this is easily shown with Conjecture 3, we instead use a weaker but unconditional result by Xylouris \cite{Xy}, the idea here being to invoke conjectures as few times as possible in our proof:

\begin{varthmC} (Xylouris 2009) For any $c$ and $d$ relatively prime, there exists a prime $p$ congruent to $c$ mod $d$ such that $$p\ll d^{5.2}.$$ In other words, for any such $c$ and $d$ there exists a $k\ll d^{4.2}$ such that $p=dk+c$ is prime.
\end{varthmC}

In our construction above, then, for a given $h$, we can find a $g_0\ll h^{4.2}$ such that $g_0h+1$ is prime.  Clearly, $h^{4.2}\ll n^5$, which is much smaller than $\left(\frac n4\right)^n$.  So for $g_0\ll h^{4.2}$, let us choose sets $S_1,\cdots,S_{g_0}\in C_h$.  Then
$$n=P\cdot \prod_{p\in S_1\bigcup\cdots \bigcup S_{g_0}}p\equiv 1\pmod{Lk_0k_1}.$$
We know that $P-1=Lk_0k_1$, and for any $p\in S_j$, $p-1|Lk_0$.  So for any $p|n$, $p-1|Lk_0k_1|n-1$.  So $n$ is a Carmichael number consisting of exactly $g_0h+1$ prime factors, and $g_0h+1$ is itself prime by assumption.  

\end{proof}
Since there are infinitely many possible choices for $y$ and $L$, there must be infinitely many Carmichael numbers with a prime number of prime factors, and hence Main Theorem 1 follows.

\section{Remarks}

We note a couple of things here:

- First, we see no imperative in this paper to try to make a sharp bound for the number of prime factors or the asymptotic density of Carmichael numbers with $\Omega(n)$ equal to some specific prime because there are so many ways that this estimate could be improved.  For instance, rather than simply taking one choice of $g$, we could find many by simply applying primes in arithmetic progressions results to $gh+1$ and being careful enough to choose an $h$ that avoids exceptional zeroes.  One could also invoke the Heath-Brown conjecture in a number of different places (in lieu of \cite{Xy} or in lieu of the estimates for primes in arithmetic progressions in the lower bounds for $\left|\mathcal P_{k_0}\right|$) which would also change the bounds rather dramatically.



- Second, a couple of other possible results fall out of this method fairly easily.  For instance, the fact that there are infinitely many Carmichael numbers with a composite number of prime factors falls out almost trivially (and unconditionally); of course, this result seems unlikely to be a surprise.  Also following rather trivially (and unconditionally) from this method is the fact that there are infinitely many Carmichael numbers where the number of prime factors is a perfect square (by taking $g=h$), perfect cube (with $g=h^2$), or, in fact, any perfect power.

\section{Acknowledgements}
I would like to thank Steven J. Miller for posing this question in the first place.  I would also like to thank the referee for a very careful reading of the paper and for many helpful comments and suggestions.

\end{document}